\documentclass{amsart}
\usepackage{graphicx}
\newtheorem{theorem}{Theorem}[section]
\newtheorem{lemma}{Lemma}[section]
\newtheorem{proposition}{Proposition}[section]
\newtheorem{definition}{Definition}[section]
\newtheorem{claim}{Claim}[section]
\newtheorem{fact}{Fact}[section]
\newtheorem{remark}{Remark}[section]
\newtheorem{example}{Example}[section]
\newtheorem{corollary}{Corollary}[section]
\newtheorem{question}{Question}[section]
\newtheorem{problem}{Problem}[section]
\numberwithin{equation}{section}

\usepackage{subfigure}

\begin{document}

\title{ On tunnel  numbers of  a cable knot and its companion}

\author{Junhua Wang}

\address{School of Mathematics and Physics, Jiangsu University of Technology, Changzhou 213001, P. R. China}

\curraddr{}
\email{jhwang@jsut.edu.cn; wangjunhua1000@163.com}


\author{Yanqing Zou}
\address{School of Mathematical Sciences \& Shanghai  Key Laboratory of PMMP,  East China Normal University}
\email{yanqing@dlnu.edu.cn; yanqing\_dut@163.com}
\thanks{This work was partially supported by NSFC No.S 11601065, 11571110, 11726609, 11726610 and Science and Technology Commission of Shanghai Municipality (STCSM), grant No. 18dz2271000. We thank Qilong Guo for his  suggestion that   we should give out a way to construct a satellite knot and its companion with arbitrary large difference between their tunnel numbers.}





\begin{abstract}
Let $K$ be a nontrivial knot in $S^{3}$ and $t(K)$ its tunnel number.  For any  $(p\geq 2,q)$-slope in  the torus boundary of  a closed regular neighborhood of $ K$ in $S^{3}$,
  denoted by $K^{\star}$, it is a nontrivial cable knot in $S^{3}$.  Though $t(K^{\star})\leq t(K)+1$,  Example \ref{example1.1} in Section 1 shows that
in some case,  $ t(K^{\star})\leq t(K)$. So it is interesting to know when $t(K^{\star})= t(K)+1$.

After using  some combinatorial techniques,  we prove that (1) for any nontrivial cable knot $K^{\star}$ and its companion $K$,  $t(K^{\star})\geq t(K)$;
(2) if either $K$ admits a high distance Heegaard splitting or $p/q$ is far away from
a fixed subset in the  Farey graph,   then $t(K^{\star})= t(K)+1$. Using the second conclusion, we construct a satellite knot and its companion so that
the difference between their tunnel numbers is arbitrary large.

\end{abstract}

\maketitle

\section{Introduction}

Let $K$ be  a nontrivial knot in $S^{3}$ and $E(K)$ its closed complement in $S^{3}$.  Then $E(K)$ admits a Heegaard splitting $V\cup_{S}W$ with $\partial E(K)=\partial_{-}W$.
Let $g(K)$ be the minimal Heegaard genus of $E(K)$  and the tunnel number $t(K)=g(K)-1$.  For any pair of pairwise coprime numbers
$p$ and $q$,  there is a slope crossing the meridian  $p$ times and the longitude $q$ times, denoted by   $p/q$, in $\partial E(K)$.   Then it is a
$(p,q)$-cable knot over $K$, denoted by $K^{\star}$, and  $K$ is a companion of $K^{\star}$.  Though  $K$ is also a cable knot of itself,  we only consider the nontrivial case and require $p\geq 2$.  Since $K^{\star}$ is contained in the closed neighborhood of $K$, it is interesting to know the difference between $t(K)$ and $t(K^{\star})$.

There is a combinatorial description of  $E(K^{\star})$ through $E(K)$,  in which way it gives an inequality between $t(K^{\star})$ and $t(K)$.  Let $\eta(K)$ be the closed regular neighborhood of $K$ in $S^{3}$. Since $K^{\star}$ is a $p/q$ slope  in $T^{2}=\partial \eta(K)=\partial E(K)$,
we can slightly push $K^{\star}$ into the interior of $\eta(K)$.  It is not hard to see that $E(K^{\star})$ is homeomorphic to the amalgamation of $E(K)$ and $\eta(K)$ along an annulus
$A$ in their common torus boundary, where the core curve of $A$ is isotopic to the $p/q$ slope.  Then $E(K^{\star})$ admits a Heegaard splitting as follows:  Let $a$ be a fiber arc in $A\times I$ from $A\times \{0\}$ to $A\times \{1\}$ and $\eta(a)$ the closed regular neighborhood in $A\times I$.  Then we define $V^{\star}= V\cup \eta(a)\cup \eta(K)$ and $W^{\star}=\overline{E(K^{\star})-V^{\star}}$. Since $W^{\star}$ is homeomorphic to  the amalgamation of a handlebody and $(T^{2}-disk)\times I$ along an once punctured annulus,  it is a compression body, see Figure 1. Moreover, $\partial _{+} W^{\star}=\partial _{+}V^{\star}=S^{\star}$.  So $V^{\star}\cup _{S^{\star}} W^{\star}$ is a Heegaard splitting of $E(K^{\star})$ and $g(K^{\star})\leq g(S^{\star})=g(S)+1$.  Therefore $g(K^{\star})\leq g(K)+1$.  Hence   $t(K^{\star})\leq t(K)+1$, see also in \cite{Moa}. However,  $t(K)+1$ is not always the best upper bound of $t(K^{\star})$, see Example \ref{example1.1} as follows.
\begin{figure}[htbp]
	\centering
	\includegraphics[width=0.80\textwidth]{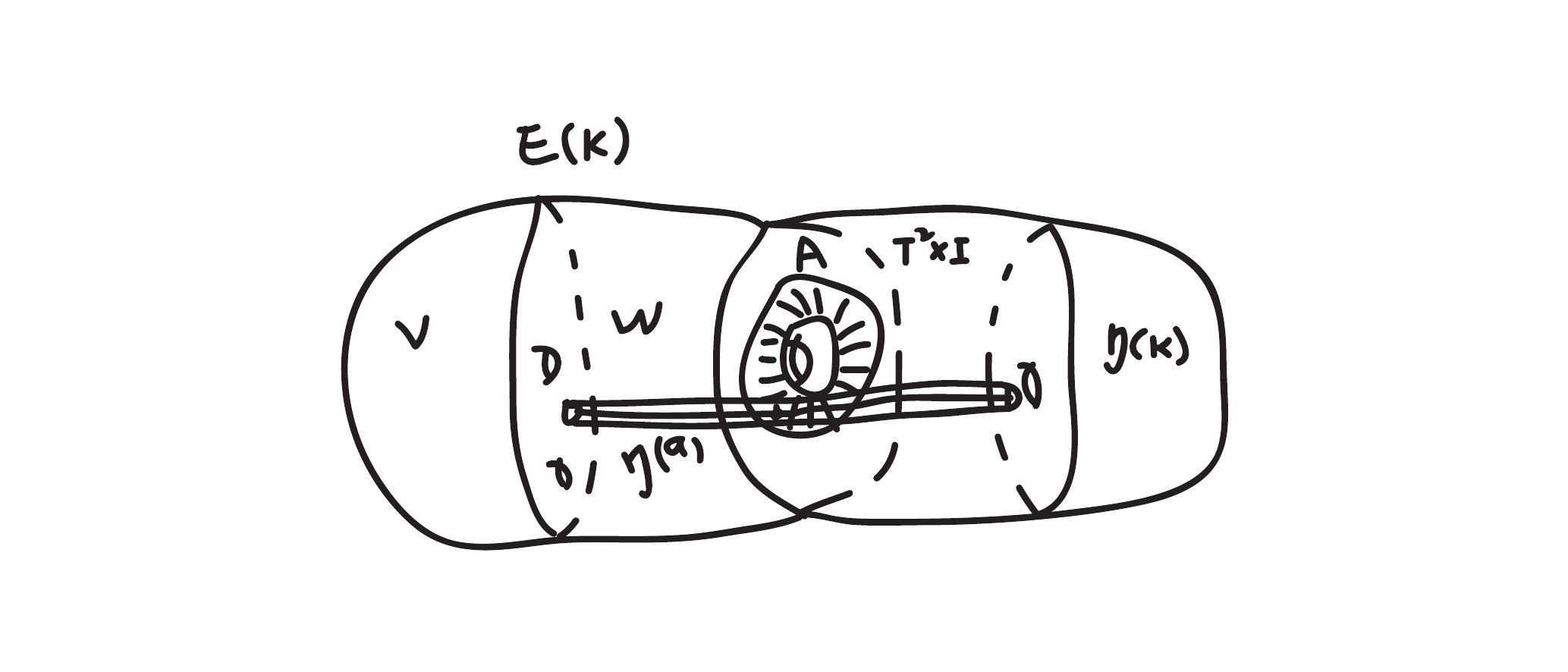}
	\caption{Construction of a Heegaard splitting $V^{\star}\cup _{S^{\star}} W^{\star}$ of $E(K^{\star})$}
\end{figure}

Before stating Example \ref{example1.1},  we introduce the definition of a $r$-primitive knot  $K$ in $S^{3}$. We assume that  the torus boundary $\partial E(K)\subset W$.  Let $r$ be a slope in $\partial E(K)$. If there is a spanning annulus $A_{1}\subset W$ with  $r\subset A$ and an essential disk $D_{1}\subset V$ such that $A_{1}\cap D_{1}$ is a single point, then  $V\cup_{S}W$ is called $r$-primitive. Moreover,  if $E(K)$ admits a $r$-primitive minimal Heegaard splitting, $K$ is called $r$-primitive.
\begin{example}
\label{example1.1}
Let $K$ be a $p/q$-primitive knot in $S^{3}$. Suppose $V\cup_{S}W$ is a $p/q$-primitive mininal Heegaard splitting of $E(K)$ with $\partial E(K)=\partial_{-}W$. Then there is a spanning annulus $A_{1}$ and an essential disk $D_{1}$ so that (1) $|\partial A_{1}\cap \partial D_{1}|=1$; (2) $\partial A_{1}\cap \partial_{-}W$ is a $p/q$ slope on $T^{2}$. In the above construction of the Heegaard splitting $V^{\star}\cup_{S^{\star}}W^{\star}$ of $E(K^{\star})=E(K)\cup_{A}\eta(K)$, let $a_{0}=a\cap W$ be the spanning arc of $A_{1}$ which is disjoint from the point $\partial A_{1}\cap \partial D_{1}$. Since $W^{\star}=\overline{W-\eta(a_{0})}\cup_{(A-disk)} [(T^{2}-disk)\times I]$, $\overline{[(\partial A_{1}-a_{0})\cap \partial_{-}W]\times I}\cup \overline{[(\partial A_{1}-a_{0})\cap \partial_{+}W]\times I}$ is an essential disk of $W^{\star}$, denoted by $D_{2}$. Note that $D_{1}$ is an essential disk in $V^{\star}$ and $|\partial D_{1}\cap \partial D_{2}|=1$. Then $V^{\star}\cup_{S^{\star}}W^{\star}$ is stabilized, see Figure 2. Hence $t(K^{\star})\leq t(K)$, see \cite{LQ}. 
\end{example}

\begin{figure}[htbp] \centering
	\subfigure[$|\partial D_{1}\cap \partial A_{1}|=1$]{\label {fig:a}
	\includegraphics[width=0.45\textwidth]{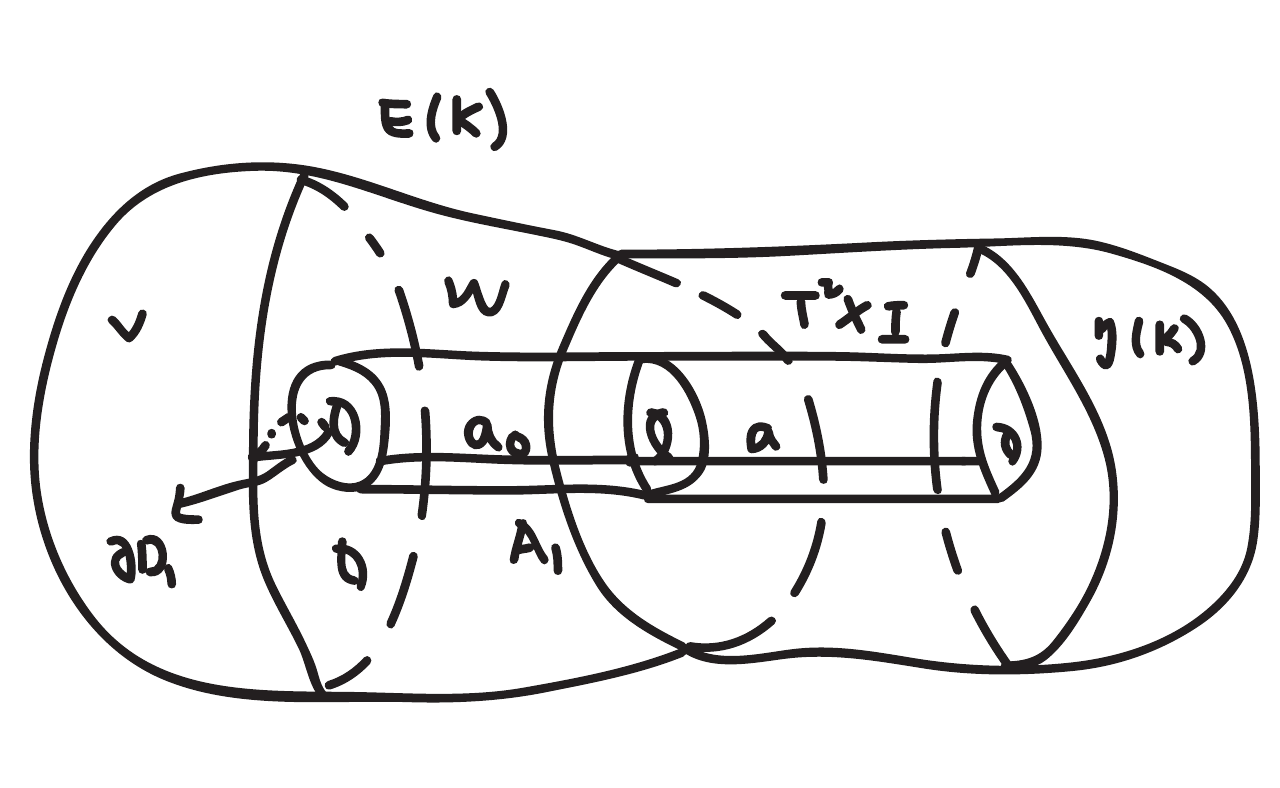}
		}
	\subfigure[$|\partial D_{1}\cap \partial D_{2}|=1$]{\label {fig:b}
	\includegraphics[width=0.45\textwidth]{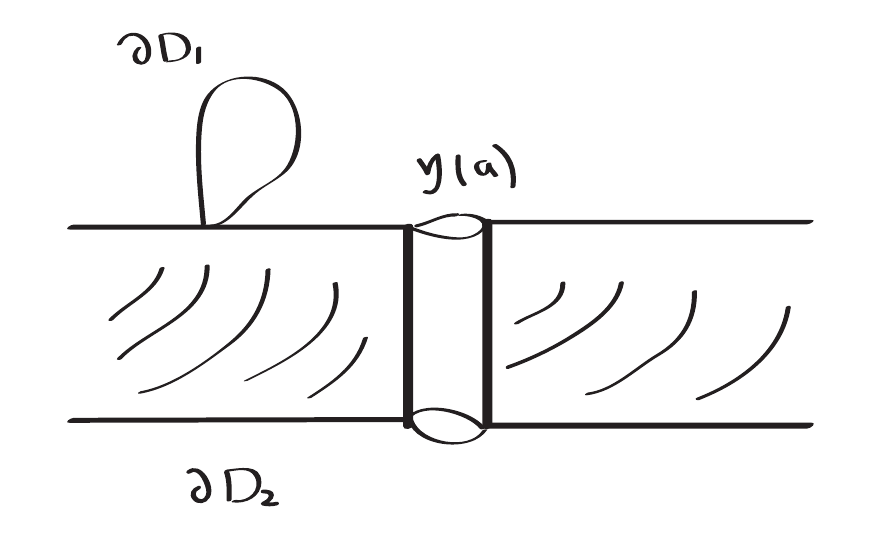}
	}
	\caption{$V^{\star}\cup_{S^{\star}} W^{\star}$ is stabilized}
\end{figure}
Then there is a natural question:  Can $\mid t(K^{\star})- t(K) \mid$   be arbitrarily large?   Unfortunately, the answer is negative.  More precisely, we have the following theorem.
\begin{theorem}
\label{theorem0}
For any nontrivial cable knot $K^{\star}\subset S^{3}$ and its companion $K\subset S^{3}$,  $t(K)\leq t(K^{\star})\leq t(K)+1$.
\end{theorem}
\begin{remark}
In a conference held in Xi'an Jiaotong University 2015, Tao Li announced that for any nontrivial satellite knot $K^{\star}\subset S^{3}$ and its companion $K\subset S^{3}$, $t(K^{\star})\geq t(K)$. But until now, we have found no preprint or publication containing this result.
\end{remark}

By Example \ref{example1.1}, if $K^{\star}$ is a (p,q)-cable knot while $K$ is $p/q$-primitive, then $t(K^{\star})\leq t(K)$.  Hence in this case, $t(K^{\star})=t(K)$. So there is  a problem as follows.

\begin{problem}
\label{problem1}
For any nontrivial cable knot $K^{\star}\subset S^{3}$ and its companion $K\subset S^{3}$,  $t(K^{\star})= t(K)$ if and only if $K^{\star}$ is a (p,q)-cable knot and $K$ is $p/q$-primitive.
\end{problem}

There are some results related to this problem. For example,  Moriah \cite{Moa} proved that if $K$ is tunnel number one but no torus knot, then $t(K^{\star})=t(K)+1$.  If we consider $E(K^{\star})$ as the amalgamation of $E(K)$ and $\eta(K)$ along an annulus,  it  is more or less similar to  the complement of  the connected sum of two knots in $S^{3}$. For any two high distance knots, Gao, Guo and Qiu \cite{GGQ} proved that $t(K_{1}\sharp K_{2})=t(K_{1})+t(K_{2})+1$.  Since a $p/q$-primitive Heegaard splitting has distance at most 2,   Example \ref{example1.1} doesn't hold on a high distance knot.  Thus we guess that if $K$ is a high distance knot,   $t(K^{\star})=t(K)+1$.

Although there are many well properties  of a high distance knot,  given a knot $K$ in $S^{3}$,  it is a little bit hard to determine whether $E(K)$ admits  a high distance Heegaard splitting or not.  So for an arbitrary nontrivial knot $K$, we  consider the problem that how to properly choose $K^{\star}$ so that $t(K^{\star})=t(K)+1$. In \cite{Li2},  Li introduced an idea to consider a sufficiently complicated gluing map in the calculation of minimal Heegaard genus of an amalgamated 3-manifold. As pointed earlier,  $E(K^{\star})$ is also an amalgamation of two 3-manifolds along an annulus. Then we guess that there should be a similar result on tunnel numbers between $K^{\star}$ and $K$ .  We present these two ideas in the following theorem.

\begin{theorem}\label{thm1}
Suppose $K^{\star}$ is a $(p\geq 2,q)$-cable knot over a nontrivial knot $K$ in $S^{3}$.
\begin{enumerate}
	\item If $E(K)$ admits a distance at least $2t(K)+5$ Heegaard splitting, then $t(K^{\star})=t(K)+1$.
	\item Let $InS$ be the collection of boundary slopes of essential surfaces properly embedded in $E(K)$.  Then there is a constant number $\mathcal{N}$ depending on $K$ so that if
	\begin{equation*}
	diam_{\mathcal{C}(T^{2})}(p/q,InS)>\mathcal{N},
	\end{equation*} then $t(K^{\star})=t(K)+1$.
\end{enumerate}
\end{theorem}
\begin{remark}
\begin{itemize}
\item In fact, there are infinitely many high distance knots in $S^{3}$, see \cite{MMS};
\item By the second conclusion in Theorem \ref{thm1}, for any given number $n\in N^{+}$, we can construct a satellite knot $K^{\star}\subset S^{3}$ and its companion $K$ so that
$t(K^{\star})-t(K)\geq n$ as follows. For any $K\subset S^{3}$, let $K^{1}$ be the cable knot so that $t(K^{1})=t(K)+1$. And then let $K^{2}$ be the cable knot of $K^{1}$ so that $t(K^{2})=t(K^{1})+1$.   It is not hard to see that $K^{2}$ lies in the solid torus neighborhood of $K$. By the same  argument again and again,  we have a knot $K^{n}$ contained in the solid torus neighborhood of $K$ and $t(K^{n})=t(K)+n$.
\end{itemize}
 \end{remark}

A Heegaard splitting $V_{1}\cup_{S_{1}}W_{1}$ is a Dehn surgery of $V\cup_{S}W$ if there is an embedded simple closed curve $c$ in $V$ or $W$ so that
 $V_{1}$ or $W_{1}$ is a Dehn surgery of $V$ or $W$ along $c$.   Under the second condition in Theorem \ref{thm1},  there is no essential surface in $E(K)$ with $p/q$ slopes as its boundary slopes. So this phenomenon derives out the following corollary.
 \begin{corollary}
 \label{cor1}
 Ler $K$, $K^{\star}$ be the same as in Theorem \ref{thm1}.  Suppose $K^{\star}$ satisfies the  condition in  Theorem \ref{thm1} (2). Then each unstabilized Heegaard splitting of
 $E(K^{\star})$ is a Dehn surgery of   one  of $E(K)$'s.
 \end{corollary}

The conclusion of Corollary \ref{cor1} implies that each unstabilized Heegaard splitting of  $E(K^{\star})$ is a Dehn surgery of   one  of $E(K)$'s. In reverse, doing a Dehn surgery on each unstabilized Heegaard splitting of $E(K)$ also produces a Heegaard splitting of $E(K^{\star})$. Then there is a natural question:

\begin{question}
For any two non isotopic but same genus Heegaard splittings of $E(K)$, is it possible that doing  two Dehn surgeries on them simultaneously produce two isotopic Heegaard splittings
of $E(K^{\star})$?
\end{question}
\begin{remark}
To our best knowledge, there is no evidence showing that whether it is true or false.
\end{remark}

We will introduce some  lemmas in Section 2,  prove Theorem \ref{theorem0} in Section 3, Theorem \ref{thm1}  in Section 4 and Corollary \ref{cor1} in Section 5.

\section{Preliminaries}

Let $V\cup_{S}W$ be a Heegaard splitting. It is stabilized if there is a pair of essential disks in $V$ and $W$ individually so that their boundary curves intersect in one point. Otherwise, it is unstabilized.  A  Heegaard splitting $V\cup_{S}W$ is reducible if there is an essential simple closed curve in $S$ bounding
a disk in $V$ and also a disk in $W$.  Otherwise, it is irreducible.  For any irreducible Heegaard splitting, Casson and Gordon \cite{CG} introduced a weakly reducible and irreducible Heegaard  splitting as follows: $V\cup_{S} W$ is weakly reducible and irreducible if it is irreducible and there are a pair of disjoint essential disks in $V$ and $W$ individually.  Otherwise, it is strongly irreducible.  So  if $V\cup_{S}W$ is unstabilized,   it is either strongly irreducible or  weakly reducible and irreducible.

If $S$ is a closed orientable and genus at least two surface, then there is  a curve complex $\mathcal {C}(S)$ defined on it as follows, see also in \cite{Ha}. The vertices are the isotopic classes of essential simple closed curves in $S$. A $n$-simplex is a set of $n+1$ isotopic classes of  nonisotopic and pairwise disjoint essential simple closed curves in $S$. Without further notation, we abuse a curve and its isotopic class. Then there is a distance defined on the one-skeleton $\mathcal{C}^{1}(S)$. Let $\alpha$ and $\beta$ be arbitrary two essential simple closed curves on $S$. The distance of $\alpha$ and $\beta$, denoted by $d_{\mathcal{C}(S)}(\alpha,\beta)$, is the minimal number of edges connecting $\alpha$ and $\beta$ in $\mathcal{C}^{1}(S)$, i.e., the minimal integer $n$ satisfying $\alpha_{0} \rightarrow \alpha_{1} \rightarrow \cdots \rightarrow \alpha_{n}$ is an edge path in $\mathcal{C}^{1}(S)$ where $\alpha_{0}=\alpha$ and $\alpha_{n}=\beta$ and $\alpha_{i} \rightarrow \alpha_{i+1}$ is an edge (i.e.,~$\alpha_{i}$ is distinct and disjoint from $\alpha_{i+1}$) $(i=0,1,\cdots,n-1)$. It is not hard to see that linearly extending  $d$ to the whole
$\mathcal {C}^{1}(S)$ is a metric of it. It is well known that $(\mathcal {C}^{1}(S), d)$ is gromov hyperbolic, see \cite{MM}.  If $S$ is a torus,  there is also a curve complex  defined on it, which is the Farey graph $\mathcal{C}(T^{2})$.  Similarly, the vertices are the isotopic classes of essential simple closed curves in $T^{2}$. We put  an edge between two isotopic classes of essential, non isotopic, intersecting one point simple closed curves in $T^{2}$.  Then for any two vertices $\alpha$ and $\beta$ in $\mathcal{C}(T^{2})$, the distance $d_{\mathcal{C}(T^{2})}(\alpha,\beta)$ is defined to be the minimal number of edges from  $\alpha$ to $\beta$ in $\mathcal{C}(T^{2})$. It is well known that the diameter of $\mathcal{C}(T^{2})$ is infinite, i.e., $diam(\mathcal{C}(T^{2}))=\infty$.

 There is a disk complex defined on either of $V$ and $W$ as follows.  The vertices are the isotopic classes of essential simple closed curves in $S$ which bound essential disks in $V$ (resp. $W$). A $n$-simplex is a set of $n+1$ isotopic classes of  nonisotopic and pairwise disjoint essential simple closed curves in $S$ which bound essential disks in $V$ (resp. $W$). Denote the disk complex of $V$ (resp. $W$) by $\mathcal {D}(V)$ (resp. $\mathcal {D}(W)$). It is not hard to see that $\mathcal {D}(V) \subset \mathcal {C}(S)$ and $\mathcal {D}(W) \subset \mathcal {C}(S)$.  Hempel \cite{He} defined the Heegaard distance $d_{\mathcal{C}(S)}(V, W)$ to be  the distance between $\mathcal {D}(V)$ and $\mathcal {D}(W)$ in $\mathcal {C}(S)$ and proved that $V\cup_{S}W$ is weakly reducible and irreducible if and only if $d_{\mathcal{C}(S)}(V, W)=1$;   $V\cup_{S}W$ is strongly  irreducible if and only if $d_{\mathcal{C}(S)}(V, W)\geq 2$.

Let $Q$ be a properly embedded essential surface,  i.e.,  incompressible and $\partial$-incompressible,  in a compact orientable 3-manifold $M$.  Hartshorn \cite {H} and Schalemann \cite{Sc}  proved that $d_{\mathcal{C}(S)}(V, W)\leq 2-\chi(Q)$. Later, Scharlemann and Tomova \cite{STo}, Li \cite{Li1} extended this result into a general case.

\begin{definition}
	Let $P$ be a separating surface properly embedded in $M$ and $\overline{M-P}=X\cup Y$.
	Suppose $P$ has compressing disks on both sides. We say $P$ is strongly irreducible if each compressing disk in $X$ meets each compressing disk in $Y$. Otherwise, $P$ is called weakly reducible.
\end{definition}
\begin{definition}
Let $P$ be a strongly irreducible surface.  Then $P$ is called $\partial$-strongly irreducible if
	\begin{enumerate}
		\item every compressing and $\partial$-compressing disk in $X$ meets every compressing and $\partial$-compressing disk in $Y$, and
		\item  there is at least one compressing or $\partial$-compressing disk on each side of $P$.
	\end{enumerate}
\end{definition}

In Li's proof of Theorem 1.1 in \cite{Li1}, it contains the following lemma:
\begin{lemma}\label{lem2}
	Suppose $M$ is a compact orientable irreducible 3-manifold and $P$ is a separating strongly irreducible surface in $M$. Let $Q$ be a properly embedded compact orientable surface in $M$ and suppose $Q$ is either essential or separating strongly irreducible. Then either
	\begin{enumerate}
		\item $d(P)\leq 2-\chi(Q)$, or
		\item after isotopy, $P_{t}\cap Q=\emptyset$ for all $t$, where $P_{t}~(t\in[0,1])$ is a level surface in a sweep-out for $P$, or
		\item $P$ and $Q$ are isotopic.
	\end{enumerate}
\end{lemma}

If $\partial M\neq \emptyset$, there is at least one essential surface with  boundary curves.  When $\partial M $ are some tori, Hatcher \cite{Hat} studied all boundary curves of essential surface in $M$.
In particular, if $\partial M$ is a torus,  there is  a finiteness result about boundary slopes of essential surfaces in $M$ as follows. For a general case,  there are some results on it,  see \cite{HRW, HWZ, LQW}.
\begin{lemma}[Corollary \cite{Hat}]
	\label{lem3}
	Let $M$ be a compact, orientable, irreducible 3-manifold. If $\partial M=T^{2}$, there are finitely many slopes realized by boundary curves of essential surfaces in $M$.
\end{lemma}

In particular, Bachman, Schleimer and Sedgewick \cite{BSS} proved the following lemma:
\begin{lemma}[Lemma 4.8\cite{BSS}]
\label{lem5}
 Suppose $M$ is compact orientable irreducible 3-manifold with a  torus boundary. 
 Let $P$ be a separating, properly embedded, connected surface in $M$ which is strongly irreducible, has non-empty boundary, and is
not peripheral. Then either $P$ is $\partial$-strongly irreducible or $\partial P$ is at most distance one
from the boundary of some properly embedded surface which is both incompressible
and boundary-incompressible.
\end{lemma}

In general, for any essential surface and any strongly irreducible and $\partial$-strongly irreducible surface in $M$,  Li proved that there is an upper bound of distances between any two of their boundary curves in $\mathcal  {C}(\partial M)$.

\begin{lemma}[Lemma 3.7 \cite{Li2}]
	\label{lem4}
	Suppose $M$ is not an $I$-bundle and has a connected boundary. Let $C_{g}$ be the collection of orientable surfaces properly embedded in $M$ with genus no more than $g$ and boundary is essential in $\partial M$. Let $P$ and $Q$ be surfaces in $C_{g}$. Suppose $Q$ is essential and suppose $P$ is either essential or strongly irreducible and $\partial$-strongly irreducible. Then there exists a number $K^{'}$ that depends only on $g$, such that the distance $d(\partial P,\partial Q)\leq K^{'}$ in $\mathcal{C}(\partial M)$.
\end{lemma}

At the end of this section, we introduce  a  lemma about essential annuli and disks in a compression body.
\begin{lemma}[Lemma 3.1 \cite{Mom}]
	\label{lem1}
	Let $V$ be a nontrivial compression body, and let $\mathcal{A}$ be a collection of pairwise  disjoint essential annuli properly embedded in $V$. Then there is an essential disk properly embedded in $V$ disjoint from $\mathcal{A}$.
\end{lemma}

\section{$t(K^{\star})\geq t(K)$}

Let $V^{\star}\cup_{S^{\star}}W^{\star}$ be  a minimal genus Heegaard splitting of $E(K^{\star})$.  Since $E(K^{\star})$ is irreducible, $V^{\star}\cup_{S^{\star}}W^{\star}$ is either strongly irreducible or weakly reducible and irreducible.

\medskip
\subsection{$V^{\star}\cup_{S^{\star}}W^{\star}$ is strongly irreducible.}  Since $A$ is an essential annulus in $E(K^{\star})$,  by Schultens  lemma \cite{Sch},  $ S^{\star}\cap A$ consists of nonzero and finitely many essential simple closed curves in both of them.  Under this condition, we assume that $ \mid  S^{\star}\cap A\mid$ is minimal.

\begin{claim}
\label{claim5.1}
One component of $\overline{S^{\star}-A}$ is strongly irreducible while the others are incompressible in $E(K^{\star})-A$.
\end{claim}

\begin{proof}
 Not only $A\cap V^{\star}$ but also $A\cap W^{\star}$   consist of finitely many essential annuli. For if not, one of them contains at least one boundary parallel annulus. Then we do an isotopy on $A$ to reduce $\mid S^{\star}\cap A\mid$, which is against  the minimal assumption of $\mid S^{\star}\cap A\mid$.  By Lemma \ref{lem1},  there is an essential disk $D \subset \overline {V^{\star}-A}$ (resp.  $E\subset \overline {W^{\star}-A}$ ). Then $\partial D$  is contained in one subsurface of $S^{\star}-A$, says $S_{1}$.  Since $V^{\star}\cup_{S^{\star}} W^{\star}$ is strongly irreducible, $\partial E\subset S_{1}$. Then any essential disk of $V^{\star}$ disjoint from $A\cap V^{\star}$  has its boundary curve in $S_{1}$. So any essential disk of $W^{\star}$ disjoint from $A\cap W^{\star}$  has its boundary curve in $S_{1}$.  Hence $S_{1}$ is the only compressible subsurface in $\overline{S^{\star}-A}$.  Since $S_{1}\subset S^{*}$ is essential, i.e., the inclusion map on  its fundamental group  is injective, $S_{1}$ is strongly irreducible.

\end{proof}

Let $S_{1}$ be the strongly irreducible  surface of  $\overline{S^{\star}-A}$. Recall that $E(K^{\star})=E(K)\cup_{A}\eta(K)$. Since its interior is disjoint from $A$,   $S_{1}$ lies in either $E(K)$ or the solid torus $\eta(K)$, abbreviated by $J$. If
$S_{1}\subset E(K)$,  then $S^{\star}\cap J$ consists of finitely many nested annuli,  denoted by $\{A_{1},..,A_{n}\}$, for some $n\in N^{+}$.

For each $1\leq i\leq n$, $\partial A_{i}\cap A$ is a pair of essential curves.  Then there is  a $1\leq j\leq n$ so that  $\partial A_{j}\cap A$  is the  innermost in $A$  which bounds an
annulus $A_{0}\subset A$.  So $A_{0}$ is an essential annulus in one compression body, says in $V^{\star}$ for example.  For if not,  we do an isotopy on $A$ to reduce $\mid S^{\star}\cap A\mid$. On one side,  since $\partial A_{0}=\partial A_{j}$,  $\partial A_{0}$ bounds an annulus in $S^{\star}$.  It means that there is a pair of two isotopic essential simple closed curves in $S^{\star}$ bounding an  essential annulus in $V^{\star}$. By the standard outermost disk argument, there is a boundary compression in $A_{0}$  producing an essential disk $D_{0}$. It is not hard to see that $D_{0}$ is separating and cuts out a solid torus $ST$ in $V^{\star}$. On the other side,  $A_{0}\cup A_{j}$ also bounds a solid torus in $J$, denoted by $ST_{0}$.  Moreover, $ST_{0}\subset ST$.   

Let $l$ be the longitude of $ST_{0}$. Then we push it a little into the interior of $ST_{0}$. For simplicity, it is still denoted by $l$. Removing  a regular neighbor of $l$, denoted by $\eta(l)$, in $ST_{0}$ and $ST$ makes $ST$ into a torus I-bundle, where the disk $D_{0}$ is in one of its boundary surface.  Since $V^{\star}$ is cutten into the solid torus $ST$ and a genus less one compression body or handlebody,   $V^{\star}-\eta(l)$ is still a compression body but with one more negative boundary surface.  In $ST-\eta(l)$, we attach a 2-handle along an essential simple closed curve in $\partial \eta(l)$   and  a 3-ball to cancel the resulted 2-sphere  so that the resulted solid torus is actually the
$A_{0}\times I $ and also $A_{j}\times I$.   Similarly we attach a 2-handle addition along the same essential simple closed curve on $V^{\star}-\eta(l)$.  So it produces a new compression body
$V$,  where $\partial _{+}V=S^{\star}$.  Then $V\cup_{S^{\star}}W^{\star}$ is still a Heegaard splitting.
For simplicity,  we replace $S^{\star}$ by $S$, $W^{\star}$ by $W$. Then $V\cup_{S}W$ is a Heegaard splitting.   As we do the Dehn surgery in $ST_{0}\subset J$,
the 3-manifold $M=V\cup_{S}W$ is an amalgamation of $E(K)$ and a solid torus along $A$. Moreover,

\begin{claim}
\label{clm5.2}
$V\cup_{S}W $ is a Heegaard splitting of $E(K)$.
\end{claim}
\begin{proof}
Since $S^{\star} \cap J$ consists of finitely many nested annuli in $J$ which are not parallel to $A$,  by doing the Dehn surgery in $V^{\star}$,  the solid torus $ST_{0}\subset J$ bounded by
$A_{0}\cup A_{j}$ is changed into a new solid torus so that $A_{j}$ is parallel to $A_{0}$.  It means that among of all these annuli $\{A_{1},..,A_{n}\}$,  each one is parallel to $A_{0}$. So we do an isotopy on $S$ so that it is disjoint from $A$.  Therefore, $A$ is in either $V$ or $W$.  Since $A$ has the same core curve with $A_{0}$,  $A$ is incompressible.  Therefore,
$A$ is boundary parallel in $V$ or $W$. So  $A$ cuts out the I-bundle $A\times I$ in $M$.  It means that $M$ is homeomorphic to $E(K)$.
\end{proof}
By Claim \ref{clm5.2},  $g(S^{\star})=g(S)$. Since $g(S^{\star})$ is a minimal Heegaard genus of $E(K^{\star})$, $g(K^{\star})\geq g(K)$. So $t(K^{\star} )\geq t(K)$.

Otherwise, $S_{1}$ lies  in the solid torus $J$.  Recall that a strongly irreducible surface is bicompressible, i.e., being compressed in its two sides, and weakly incompressible,  i.e.,  no disjoint compression disks from its two sides. Scharlemann\cite{scharlemann98} studied the bicompressible but weakly incompressible surfaces in a solid torus, proved the following lemma.
\begin{lemma}[Proposition 3.2\cite{scharlemann98}]
\label{lemma5.1}
Let $c$ be essential simple closed curve in $\partial J$ so that it  neither bounds a disk in $J$ nor intersects the essential disk in $J$ in one point.
Then for any bicompressible, weakly incompressible surface with $c$ as its boundary curve,  it is either a boundary parallel  incompressible annulus  with  a tube parallel to an arc in $\partial J$ attached or the tube sum of two boundary
parallel incompressible annuli and the tube is parallel to an arc in $\partial J$.
\end{lemma}
So $S_{1}$ is either a boundary parallel annulus with a tube attached  or the tube sum of two boundary parallel annuli in $J$.   In case of a long argument,  we divide its proof into these two lemmas \ref{lemma5.2} and \ref{lemma5.3}.

\begin{lemma}
\label{lemma5.2}
If  $S_{1}$ is  a boundary parallel annulus $A^{'}$ with a tube attached in $J$,  then $t(K^{\star})\geq t(K)$.
\end{lemma}
\begin{proof}
Since $\partial A^{'}$ bounds a subannulus $A^{''}\subset  A$, $A^{'}$ is  parallel to either $A^{''}$ or its complement annulus in $\partial J$.

(1) If $ A^{'}$ is parallel to  a subannulus $A^{''}\subset  A$, then $int(A^{''})\cap S^{\star}=\emptyset$ up to isotopy, see Figure 3. For if not, then there is an annulus in $S^{\star}\cap J$ parallel to $A$. 
So we do an isotopy along this annulus to reduce $\mid S^{\star}\cap A\mid$.  Then $A^{''}$ is an essential annulus in one of $V^{\star}$ and $W^{\star}$. Without loss of generality,  we assume that $A^{''}\subset V^{\star}$.  The other case is similar. So we omit it.  Then cutting $V^{\star}$ along $A^{''}$
 produces two handlebodies  or one handlebody and a compression body. Let $V^{'}$ be the compression body or the handlebody  containing no $S_{1}$.

\begin{figure}[htbp]
	\centering
	\includegraphics[width=0.80\textwidth]{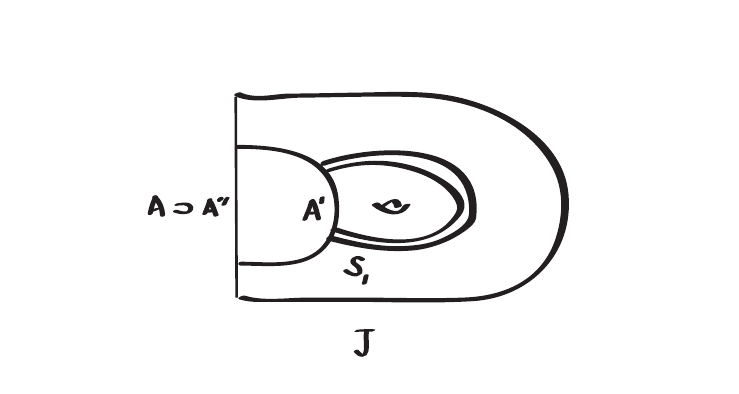}
	\caption{}
\end{figure}

By Claim \ref{claim5.1},  except $S_{1}$, all other components of $S^{\star}\cap J$ are nested annuli in $J$. If there is no annulus in $S^{\star}\cap J$,   there are a pair of two spanning annuli,  says
$A^{2''}$ and $A^{3''}$ so that  one component of  $W^{\star}-A^{2''}\cup A^{3''}$, says $N$, lies in $E(K)$. It is not hard  to see that $N$ is a handlebody and $\partial A^{''}\subset N$.  Let $A^{''}\times I$ be I-bundle in $V^{'}$.  Then we  remove this I-bundle from $V^{'}$ and  attach  it to $N$ along these two annuli $\partial A^{''}\times I$.  So  $N$  is changed into a
new 3-manifold, says $N_{0}$.  In this case, $\partial N_{0}$ contains $\partial E(K)$.  Let $a$ be a vertical arc connecting $A^{''}\times \{0,1\}$ in $A^{''}\times I$.  So $a$ is a fiber arc
in $N_{0}$ connecting $\partial E(K)$ and  the other boundary surface of $N_{0}$.  We remove the union of a regular neighborhood of $\partial E(K)$ and a regular neighborhood of $a$ in $N_{0}$ from $N_{0}$. Then $N_{0}$ is changed into a 3-manifold $W$, which is also the amalgamation of $N$ and a tubed  annulus I-bundle along two essential annuli.  It is not hard to see that $W$ is also a handlebody and furthermore $\overline{E(K)-W}$ is  the  disk sum of $V^{'}$ and $\partial E(K)\times I$,  denoted by $V$.  So $V$ is a compression body and $V\cup_{\partial _{+}V=\partial W} W$ is a Heegaard splitting of $E(K)$.  During the process of this surgery,   $g(\partial W)=g(S^{\star})$. It means that
$g(K^{\star})\geq g(K)$ and $t(K^{\star})\geq t(K)$.

So we assume that there is at least one annulus in $S^{\star}\cap J$.  Let $A^{2}$ be the innermost annulus in $S^{\star}\cap J$, i.e., no other component of $S^{\star}\cap J$ lies between it and $S_{1}$.  Then $A^{2}$ is not boundary parallel into $A$.  For if not, then we can do an isotopy on $A$ to reduce $S^{\star}\cap J$. Let $V^{'}$ be  the handlebody or compression body as above.  Then $\overline{E(K^{\star})-V^{'}}$ contains a smaller copy of $J$.  Then we do a Dehn surgery on this smaller copy of $J$ along its longitude so that  $A^{2}$ is parallel to  $A^{'}$. Then  $\overline{E(K^{\star})-V^{'}}$ is changed  into a new 3-manifold, denoted by $N$ again.  By the same argument  of
 the case that $S^{\star}\cap J$ is incompressible,   $N\cup V^{'}$  is homeomorphic to $E(K)$.   Since $A^{2}\cup A^{'}$ bounds an annulus I-bundle,  let $a$ be a fiber arc connecting
 $A^{2}$ and $A^{'}$ in this I-bundle.  Then we remove  a regular neighborhood of $a$ and then $N$ is changed into a 3-manifold $W$.  By the same argument as above,  $W$ is a compression body or handlebody.  Without loss of generality, we assume that $W$ is a handlebody. Meanwhile, the union of $V^{'}$ and the closed regular neighborhood of $a$ is a  compression body. So  $V\cup_{\partial _{+}V=\partial W} W$ is a Heegaard splitting of $E(K)$. By the same argument,  $t(K^{\star})\geq t(K)$.

 (2)  If $A^{'}$ is parallel to the complement annulus of $A^{''}$ in $\partial J$, then $int(A^{''})\cap S^{\star}=\emptyset$.  For if not,  by Claim \ref{claim5.1},  then there is an innermost annulus $A{''''}$ of $S^{\star}\cap J$  so that its boundary curves lie in the interior of $A^{''}$.  By minimality of $\mid S^{\star}\cap A \mid$, $A^{''''}$ is not boundary parallel to $A$.  Then $A^{''''}\cup A^{''}$  bounds a smaller copy of $J$.  By the similar argument as above, 
 we do a Dehn surgery on this smaller copy of $J$ and get a Heegaard splitting of $E(K)$.  Moreover,  $t(K^{\star})\geq t(K)$.


It is not hard to see that the  dual core curve of the tube in $S_{1}$ bounds an essential disk $E$, says in $W^{\star}$. Then we do a compression on $W^{\star}$ along the disk $E$ and get a new 3-manifold $W^{'}$.  Moreover, $S_{1}$ is changed into an annulus $A^{'}$ in $J$.  Then $A^{'}\cup A^{''}$ bounds a smaller copy of $J$ in $\overline{E(K^{\star})-W^{'}}$. Denote $N=\overline{E(K^{\star})-W^{'}}$.
If  $A^{''}$ is an essential annulus in $V^{\star}$, then $N$ is an amalgamation of a handlebody or a compression body $V^{'}$ and a smaller copy of $J$ along $A^{''}$.   So we do a Dehn surgery as above so that  the smaller copy of $J$ is changed into  the I-bundle $A^{''}\times I$ and $N$ is changed into a 3-manifold $V$, where $V$ is the union of $V^{'}$ and $A^{''}\times I$. Then $V$ is homeomorphic to $V^{'}$.  By the same argument as above,  $V^{'}\cup W$ is a Heegaard splitting of $E(K)$ and $t(K^{\star})\geq t(K)+1$.

Otherwise, $A^{''}$ is an essential annulus in $W^{\star}$.  Then both of $A^{'}$ and $A^{''}$ are in $W^{\star}$, where they bound a smaller copy of $J$. Then we do a Dehn surgery on  this smaller copy of $J$ so that they are parallel. 
Hence $W^{\star}$ is changed into a compression body or handlebody $W$. Replace $V^{\star}$ by $V$.  So $V\cup_{S^{\star}}W$ is a Heegaard splitting of $E(K)$ and $t(K^{\star})\geq t(K)$.
\end{proof}

\begin{lemma}
\label{lemma5.3}
If $S_{1}$ is the tube sum of two annuli in $J$, then $t(K^{\star})\geq t(K)$.
\end{lemma}

\begin{proof}
Let $S_{1}$ be  the tube sum of two annuli $A_{1}$ and $A_{2}$ in $J$. We say there  is no other component of $S^{\star}\cap J$  in the region between $A_{1}$ and $A_{2}$ in $J$.  For if not, then the tube would not connect them.   It is not hard to see that $\partial A_{1}$ (resp. $\partial A_{2}$) bounds an annulus in $A$. Without loss of generality, we assume that the annulus  bounded by $\partial A_{1}$ in $A$ doesn't contain $\partial A_{2}$.  Then we say that there is no other component of $S^{\star}\cap J$,   of which boundary curves lies in the annulus  bounded by $\partial A_{1}$ in $A$. Otherwise, we either  do  an isotopy on $A$ to reduce $\mid S^{\star}\cap A\mid$ or do a Dehn surgery as above and get $t(K^{\star})\geq t(K)$. 

There are three  types of $A_{1}$ and $A_{2}$ as follows: 
\begin{itemize}
\item  (1) one of them is parallel to $A$, says $A_{1}$ while $A_{2}$ not; 
\item  (2)  both of them are parallel to $A$; 
\item  (3)  neither of them is parallel to $A$. 
\end{itemize}
For  the first case, either  $\partial A_{1}$ separate $\partial A_{2}$ in $A$ or not.  If $\partial A_{1}$ separate $\partial A_{2}$,  then  $A_{1} \cup A_{2}$ bounds a smaller copy of $J$. 
 Then we do a  compression on this tube along the disk $D$ bounded by the core curve.  Without loss of generality, we assume that $D\subset V^{\star}$. So $V^{\star}$ is changed into a new 3-manifold, denoted by $V^{'}$, which is also a compression body or handlebody.  Denoted $\overline{E(K^{\star})-V^{'}}$ by $N$. Again we do a Dehn surgery on $N$ as in the proof of Lemma \ref{lemma5.2} and so the smaller copy of $J$ is changed into the annulus I-bundle $A_{1}\times I$.  So $N$ is changed into $N^{'}$ and $N^{'}\cup V^{'}=E(K)$. Let  $a$ be the vertical arc of $A_{1}\times I$, which connects $A_{1}\times \{0,1\}$.  Then $W=N^{'}-N(a)$ is a compression body or handlebody.  Meanwhile, the union of $V^{'}$ and the closed regular neighborhood of $a$ is a handlebody or compression body, denoted by $V$.  Moreover, $g(\partial_{+} V)=g(S^{\star})$. So $t(K^{\star})\geq t(K)$. If $\partial A_{1}$ doesn't separate $\partial A_{2}$,  then there is a smaller copy of $J$ bounded by $A_{2}$. So we do a dehn surgery as above.  Then $t(K^{\star})\geq t(K)$.
 
 For the second case,  $A_{1}$ and $A_{2}$ are not nested.   Without loss of generality, we assume that the disk $D$ bounded by core curve of this tube lies in $V^{*}$.  Then we do a
 compression on $S_{1}$ along an essential disk $E$ in $W^{*}$.  So $W^{*}$ is changed into a new compression body or handlebody, says $W^{'}$.  Since $\partial A_{1}$ and $\partial A_{2}$ bound two essential annuli in $V^{*}$,  cutting $V^{*}$ along them produce a handlebody or compression body, says $V^{'}$. Then $E(K^{*})-W^{'}$ is the amalgamation of 
 $V^{'}$ and an annulus I-bundle or a smaller copy of $J$. For the first case, removing a closed neighborhood of a fiber arc in this I-bundle changes $E(K^{*})-W^{'}$ into a handlebody or compression body $V$.  For the later case, we do a dehn surgery on  this smaller copy of $J$ so that it is changed into an annulus I-bundle.  In both of these two cases, $t(K^{\star})\geq t(K)$.

For the third case,  $A_{1} $ and $ A_{2}$  are parallel in $J$ and there is  a smaller copy of $J$ bounded by a subannulus in $A$ and $A_{1}$.  For if not, then both of $A_{1}$ and $A_{2}$ lie in an I-bundle of $A$.  Then we do an isotopy to reduce 
 $S^{\star}\cap A$.  We do a  compression on this tube along the disk $D$ bounded by the core curve.  Without loss of generality, we assume that $D\subset V^{\star}$. So $V^{\star}$ is changed into a new 3-manifold, denoted by $V^{'}$, which is also a compression body or handlebody.  Denoted $\overline{E(K^{\star})-V^{'}}$ by $N$. Since there is no other component of $S^{\star}\cap J$ lie between $A_{1}$ and $A$,  this smaller copy of $J$ lies in $V^{'}$.  By the same argument in the proof of Lemma \ref{lemma5.2}, we  do a Dehn surgery on this smaller copy of $J$ so that it is changed into $A_{1}\times I$.  So $V^{'}$ is changed into a new compression body or handlebody $V^{''}$. Moreover, $N\cup V^{''}=E(K)$.  On one side, we  attach a closed regular neighborhood of a vertical fiber arc $a$ in the I-bundle bounded by $A_{1}$ and $A_{2}$ to $V^{''}$ so that $V^{''}$ is changed into a compression body or handlebody, says $V$.  On the other side,  $W=\overline{N-N(a)}$ is still a handlebody or compression body and $\partial_{+}W=\partial_{+}V$. During this process,  $g(\partial_{+}V)=g(S^{\star})$ and $t(K^{\star})\geq t(K)$.
\end{proof}

\begin{figure}[htbp] \centering
	\subfigure[]{\label {fig:a}
		\includegraphics[width=0.8\textwidth]{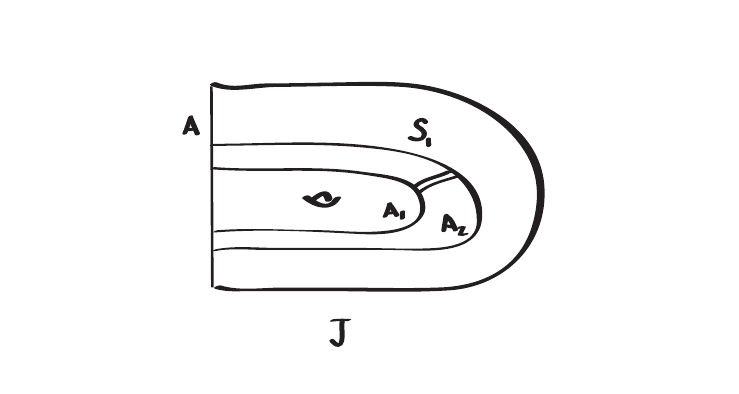}
	}
	\subfigure[]{\label {fig:b}
		\includegraphics[width=0.8\textwidth]{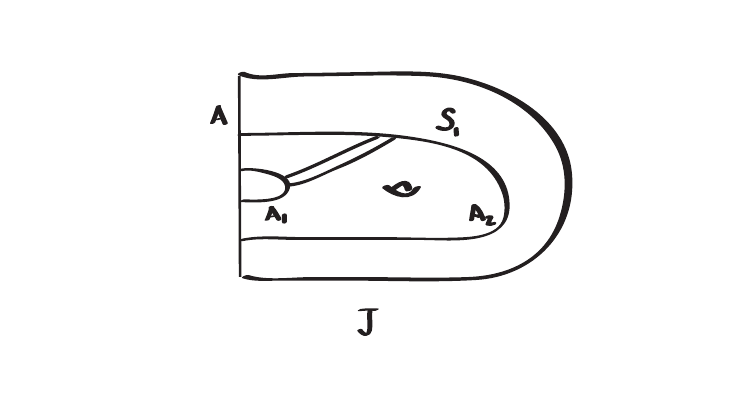}
	}
	\caption{}
\end{figure}

\subsection{$V^{\star}\cup _{S^{\star}} W^{\star}$ is weakly reducible and irreducible. } By the main result in \cite{STh},  $V^{\star}\cup _{S^{\star}} W^{\star}$ has an untelescoping,  says $(V_{1}\cup_{S_{1}}W_{1})\cup_{F_{1}}...\cup_{F_{n-1}}(V_{n}\cup_{S_{n}}W_{n})~(n\geq 2)$. We assume that each component of $[(\mathop\cup\limits_{i=1}^{n} S_{i})\cup (\mathop\cup\limits_{i=1}^{n-1} F_{i})]\cap A$  is essential in both of them up to isotopy.  
Under this condition, we assume that  $\mid [(\mathop\cup\limits_{i=1}^{n} S_{i})\cup (\mathop\cup\limits_{i=1}^{n-1} F_{i})]\cap A \mid$ is minimal. Then each component of $(\mathop\cup\limits_{i=1}^{n-1} F_{i})\cap J$ is an  incompressible  annulus  in $J$, which is not boundary parallel to $A$ in $J$. So  $(\mathop\cup\limits_{i=1}^{n-1} F_{i})\cap J$  are nested annuli in $J$. Let $A_{1}$ be  the innermost one among $ (\mathop\cup\limits_{i=1}^{n-1} F_{i})\cap J $. Since $\partial A_{1}$ bounds an annulus $A^{'}\subset A$, $A_{1}\cup A^{'}$  bounds a smaller copy of  $J$, denoted by $J^{'}$. Then $J^{'}$ lies in some $V_{i}\cup_{S_{i}}W_{i}$, for some $1\leq i\leq n$. 

By Claim \ref{claim5.1},  $S_{i}\cap J$ contains at most one strongly irreducible  surface, denoted by $S_{i,1}$.  If there is no strongly irreducible surface in $S_{i}\cap J$, by the similar argument as above, we do a Dehn surgery on $V_{i}\cup_{S_{i}}W_{i}$ so that
$J^{'}$ is changed into  the annulus I-bundle $A^{'}\times I$.  Then the Heegaard splitting $V_{i}\cup_{S_{i}} W_{i}$ is changed into $V_{i,1}\cup_{S_{i}} W_{i,1}$.  So the amalgamation
$(V_{1}\cup_{S_{1}}W_{1})\cup_{F_{1}}...\cup_{F_{i-1}} (V_{i,1}\cup_{S_{i}}W_{i,1})\cup_{F_{i}}...\cup_{F_{n-1}}(V_{n}\cup_{S_{n}}W_{n})~(n\geq 2)$ is a Heegaard splitting of $E(K)$.
Hence  $g(K^{\star})\geq g(K)$ and hence $t(K^{\star})\geq t(K)$.

So we assume that $S_{i,1}$ is strongly irreducible in $J^{'}$. By Lemma \ref{lemma5.1}, $S_{i,1}$ is either a boundary parallel annulus with a tube attached or the tube sum of two boundary parallel annuli in $J^{'}$. Though  the argument is almost same to the proofs in Lemma \ref{lemma5.2} and \ref{lemma5.3}, it is slightly different.  Without loss of generality, we  assume that
the dual disk to the tube lies in $V_{i}$.  For the first case,  we do a compression on $S_{i,1}$ along a non-separating  essential disk $D$ in $W_{i}$. On one hand,  $W_{i}$ is changed into a genus less one compression body or handlebody $W_{i,1}$.  On the other hand,   $\partial S_{i,1}$ bounds an  incompressible annulus $A^{''}\subset A^{'}$.  So $A^{''}$ is an essential  annulus in $V_{i}$. Then one component of $\overline{V_{i}-A^{''}}$ contains no $S_{i,1}$, denoted by $V^{''}$. It is not hard to see that $V^{''} $ is a  handlebody or compression body.  Since $S_{i,1}$ is a genus one, two boundary curves compact surface,  $S_{i,1}$ is changed into an annulus $A^{2}$  and $\partial A^{2}=\partial S_{i,1}=\partial A^{''}$.    It  means that the complement  of $W_{i,1}$ in  $M_{i}=V_{i}\cup_{S_{i}} W_{i}$,  denoted by $N_{i}$, is the amalgamation of $V^{''}$ and the solid torus bounded by $A^{2}\cup A^{''}$ along $A^{''}$.

It is known that  $A^{2}$ is incompressible in $J$.  Then either $A^{2}$ is parallel to $A^{''}$ or there is a smaller copy of $J^{'}$ bounded by $A^{2}\cup A^{''}$. If $A^{2}$ is parallel to $A^{''}$,  then $N_{i}$ is homeomorphic to $V^{''}$.  So it is a handlebody or compression body, denoted by $V_{i,1}$.  It means that   $V_{i,1}\cup_{\partial _{+}V_{i,1}} W_{i,1}$ is a genus $[g(S_{i})-1]$ Heegaard splitting of $V_{i}\cup_{\partial_{+}V_{i}}W_{i}$, which is impossible.
So $A^{2}\cup A^{''}$ bounds a smaller copy of $J^{'}$.  Then by the same argument in Lemma \ref{lemma5.3}, we do a Dehn surgery on $J^{'}$ in $N_{i}$ so that $N_{i}$ is changed into a compression body or handlebody  $V_{i,1}$. And   $V_{i,1}\cup_{\partial _{+}V_{i,1}} W_{i,1}$ is a genus $[g(S_{i})-1] $ Heegaard splitting. By the similar argument,   $(V_{1}\cup_{S_{1}}W_{1})\cup_{F_{1}}...\cup_{F_{i-1}} (V_{i,1}\cup_{S_{i}}W_{i,1})\cup_{F_{i}}...\cup_{F_{n-1}}(V_{n}\cup_{S_{n}}W_{n})~(n\geq 2)$ is a Heegaard splitting of $E(K)$. So $t(K^{\star})\geq t(K)+1$.

The left case is that $S_{i,1}$ is the tube sum of two boundary parallel annuli in $J^{'}$. By the same argument in Lemma \ref{lemma5.3},  $V_{i}\cup_{S_{i}} W_{i}$ is changed into  $V_{i,1}\cup_{\partial _{+}V_{i,1}} W_{i,1}$, where $g(S_{i})=g(\partial_{+} V_{i,1})$.  Moreover,  $(V_{1}\cup_{S_{1}}W_{1})\cup_{F_{1}}...\cup_{F_{i-1}} (V_{i,1}\cup_{S_{i}}W_{i,1})\cup_{F_{i}}...\cup_{F_{n-1}}(V_{n}\cup_{S_{n}}W_{n})~(n\geq 2)$ is a Heegaard splitting of $E(K)$. So  $t(K^{\star})\geq t(K)$.

\section{The proof of Theorem \ref{thm1}}

Let $K$, $K^{\star}$, $t(K)$,  $t(K^{\star})$, $T^{2}$ and $\mathcal {C}(T^{2})$ be the same as in Section 1. We rewrite Theorem \ref{thm1} as the following proposition:

\begin{proposition}
\label{pro4.1}
	Suppose $K^{\star}$ is a $(p\geq 2,q)$-cable knot over a nontrivial knot $K$ in $S^{3}$.
	\begin{enumerate}
		\item If $E(K)$ admits a distance at least $2t(K)+5$ Heegaard splitting, then $t(K^{\star})=t(K)+1$.
		\item Let $InS$ be the collection of boundary slopes of essential surfaces properly embedded in $E(K)$.  Then there is a constant  $\mathcal{N}$ depending on $K$ so that if
		\begin{equation*}
		diam_{\mathcal{C}( T^{2})}(p/q,InS)>\mathcal{N},
		\end{equation*} then $t(K^{\star})=t(K)+1$.
	\end{enumerate}
\end{proposition}

\begin{proof}
    Since $K^{\star}$ is  in $T^{2}=\partial \eta(K)=\partial E(K)$,
    we can slightly push $K^{\star}$ into the interior of $\eta(K)$. Then $E(K^{\star})=E(K)\cup_{T^{2}}C$ where $C=(T^{2}\times I)\cup_{A}\eta(K)$, see Figure 5.
    \begin{figure}[htbp]
    	\centering
    	\includegraphics[width=0.80\textwidth]{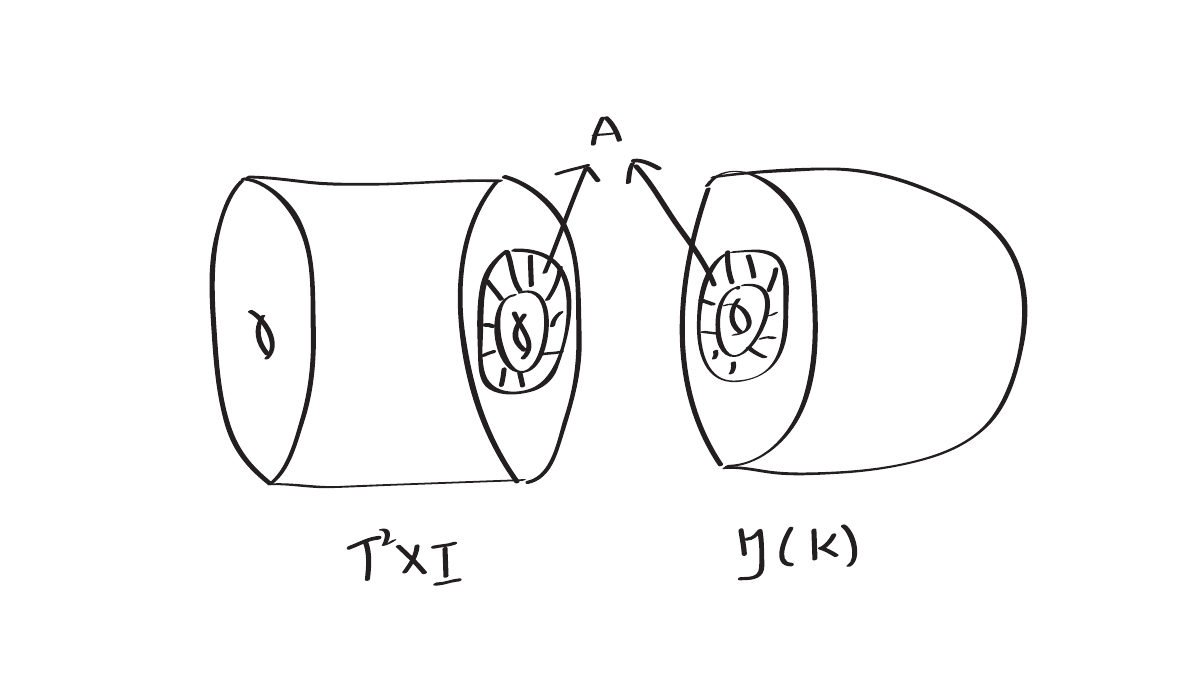}
    	\caption{$C=(T^{2}\times I)\cup_{A}\eta(K)$}
    \end{figure}
	
	On one side, since $K$ is nontrivial, $T^{2}$ is incompressible in $E(K)$. On the other side, $T^{2}$ is incompressible in $C$ and not parallel to $\partial C-T^{2}=\partial E(K^{\star})$.  Then $T^{2}$ is essential in $E(K^{\star})$. We will prove Proposition \ref {pro4.1} (1) in subsection 4.1 and \ref {pro4.1} (2) in subsection 4.2.
	
  \subsection{Proof of Proposition 4.1 (1)}
	
	Let $V^{\star}\cup_{S^{\star}}W^{\star}$ be a minimal genus Heegaard splitting of $E(K^{\star})$ with $\partial_{-}V^{\star}=\partial E(K^{\star})$. Then it is either strongly irreducible or weakly reducible and irreducible.
	
	\begin{lemma}\label{lemma1}
		$V^{\star}\cup_{S^{\star}}W^{\star}$ is weakly reducible and irreducible.
	\end{lemma}
    \begin{proof}	
    Suppose the conclusion is false. Then $V^{\star}\cup_{S^{\star}}W^{\star}$ is strongly irreducible. Since $T^{2}\subset E(K^{\star})$ is essential, then $T^{2}$ intersects $S^{\star}$ nontrivially up to isotopy. By Schultens' lemma \cite{Sch}, we  assume that (1) $|S^{\star}\cap T^{2}|$ is minimal; (2) each simple closed curve of $S^{\star}\cap T^{2}$ is essential in both $S^{\star}$ and $T^{2}$.
    \begin{claim}\label{claim1}
    	There is at most one strongly irreducible component in $S^{\star}\cap E(K)$ while others are essential in $E(K)$.
	\end{claim}
	\begin{proof}
		Since $T^{2}\cap V^{\star}$ (resp. $T^{2}\cap W^{\star}$) is a collection of disjoint essential annuli in $V^{\star}$ (resp. $W^{\star}$), by Lemma \ref{lem1}, there is a compressible disk $B$ (resp. $D$) in $V^{\star}$ (resp. $W^{\star}$) so that $B$ (resp. $D$)  disjoint from $T^{2}\cap V^{\star}$ (resp. $T^{2}\cap W^{\star}$). Since $S^{\star}$ is strongly irreducible, both $B$ and $D$ lie in $E(K)$ or $C$ . Furthermore, $\partial B$ and $\partial D$ lie in a same component $S_{1}$ of $\overline{S^{\star}-T^{2}}$. Moreover, $S_{1}$ is strongly irreducible while other components of $\overline{S^{\star}-T^{2}}$ are essential. For if not, then there is another compressible component of $\overline{S^{\star}-T^{2}}$. It means that $S^{\star}$ is weakly reducible, see Figure 6.
	\end{proof}
	
	\begin{figure}[htbp]
		\centering
		\includegraphics[width=0.60\textwidth]{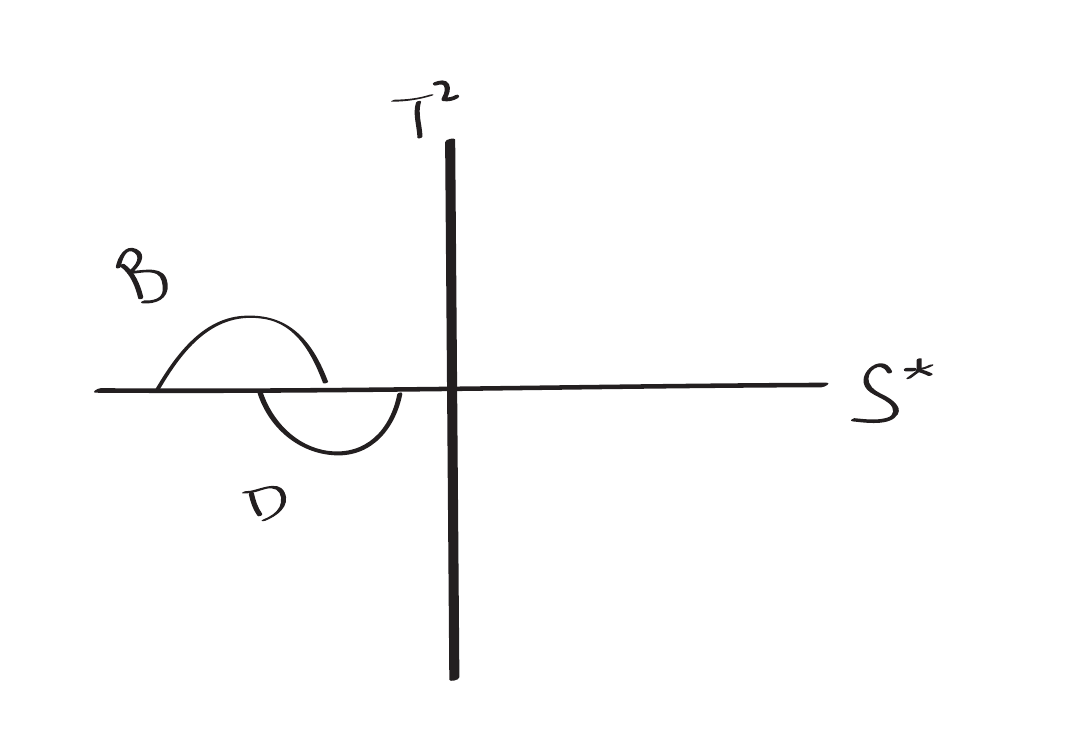}
		\caption{$S^{\star}\cap E(K)$}
	\end{figure}
	
	Let $V\cup_{S}W$ be a distance at least $2t(K)+5$ Heegaard splitting of $E(K)$. Then $S$ is strongly irreducible in both $E(K)$ and $E(K^{\star})$. By Claim \ref{claim1}, there is at most one strongly irreducible component in $S^{\star}\cap E(K)$ while others are essential in $E(K)$. If $S^{\star}\cap E(K)$ contains only one strongly irreducible component $S_{1}$, then $S_{1}$ is separating. Since $S_{1}$ intersects $T^{2}$ nontrivially up to isotopy, $S_{1}$ and $S$ are not well-separated.  Moreover, $S_{1}$ is not isotopic to $S$. Then by Lemma \ref{lem2}, $d_{\mathcal{C}(S)}(V,W)\leq 2-\chi(S_{1})$. Since $\partial S_{1}$ is essential in  $S^{\star}$, $2-\chi(S_{1}) \leq 2-\chi(S^{\star}\cap E(K)) \leq 2-\chi(S^{\star})=2g(S^{\star})=2g(K^{\star})=2t(K^{\star})+2\leq 2t(K)+4$. Then $d_{\mathcal{C}(S)}(V,W)\leq 2t(K)+4$. A contradiction. Otherwise, $S^{\star}\cap E(K)$ contains an essential subsurface  $S_{1}$ in $E(K)$. Then by the same argument, $d_{\mathcal{C}(S)}(V,W)\leq 2t(K)+4$. A contradicition.

    \end{proof}

	So $V^{\star}\cup_{S^{\star}}W^{\star}$ is weakly reducible and irreducible. By \cite{STh}, $V^{\star}\cup_{S^{\star}}W^{\star}$ has an untelescoping $(V_{1}\cup_{S_{1}}W_{1})\cup_{F_{1}}...\cup_{F_{n-1}}(V_{n}\cup_{S_{n}}W_{n})$,  for $n\geq 2$, so that  (1)  $F_{i}$ is essential  in $E(K^{\star})$, for any $1\leq i\leq n-1$; (2) $V_{j}\cup_{S_{j}}W_{j}$ is a strongly irreducible Heegaard splitting,  for any $1\leq j\leq n$,  see Figure 7.
	
	\begin{figure}[htbp]
		\centering
		\includegraphics[width=0.80\textwidth]{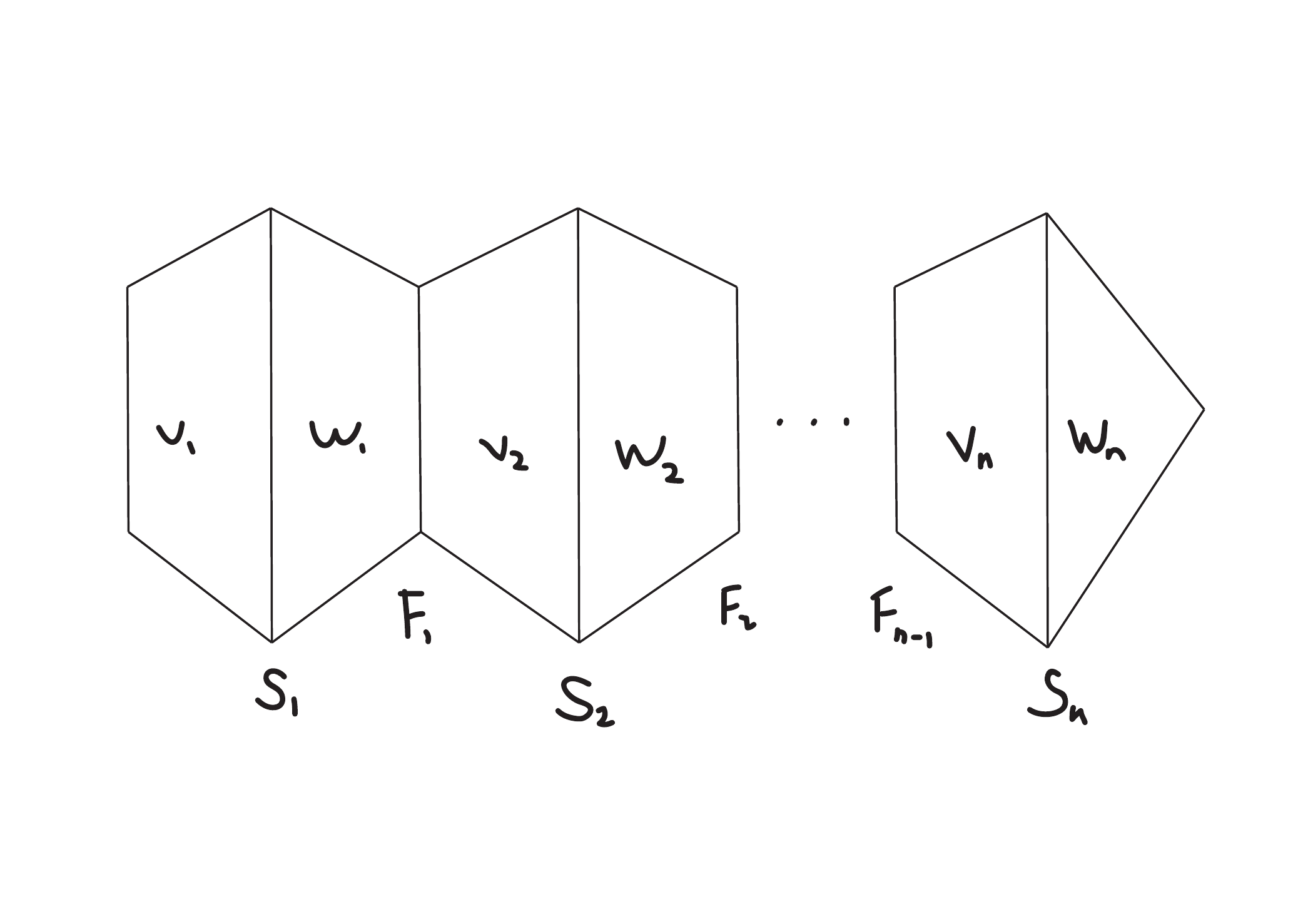}
		\caption{Untelescoping of $V^{\star}\cup_{S^{\star}}W^{\star}$}
	\end{figure}
	
	Since  \[g(S^{\star})=\sum\limits_{i=1}^{n}g(S_{i})-\sum\limits_{j=1}^{n-1}g(F_{j})=\sum\limits_{i=1}^{n-1}(g(S_{i})-g(F_{i}))+g(S_{n})\]
	
	\[=\sum\limits_{i=1}^{k-1}(g(S_{i})-g(F_{i}))+g(S_{k})+\sum\limits_{j=k+1}^{n}(g(S_{j})-g(F_{j-1}))~(2\leq k\leq n-1)\]
	
	\[=g(S_{1})+\sum\limits_{j=2}^{n}(g(S_{j})-g(F_{j-1}))\]
	and
	\[g(S_{i})\geq max~\{g(F_{i}),~g(F_{i-1})\}~(2\leq i\leq n-1),\]
	\[g(S_{1})\geq g(F_{1}),\]
	\[g(S_{n})\geq g(F_{n-1}),\]
	we have
	\[g(S^{\star})\geq g(S_{k})~(1\leq k\leq n).\]

	If  for some $1\leq i\leq n-1$, $F_{i}$ intersects  $T^{2}$  nontrivially up to isotopy,  then each component of $F_{i}\cap E(K)$ is essential in $E(K)$.  Let $F_{1,i}\subset F_{i}\cap E(K)$ be an essential surface in $E(K)$.  Then by Lemma \ref{lem2}, $d_{\mathcal{C}(S)}(V,W)\leq 2-\chi(F_{1,i})\leq 2-\chi(F_{i})=2g(F_{i})\leq 2g(S_{i})\leq 2g(S^{\star})\leq 2t(K)+4$. It contradicts the assumption that $d_{\mathcal{C}(S)}(V,W)\geq 2t(K)+5$. So  for any $1\leq i\leq n-1$, $F_{i}$  is disjoint from $T^{2}$. By the similar argument in the proof of Lemma \ref{lemma1},  for any $1\leq j\leq n$, $S_{j}$  is disjoint from $T^{2}$. Hence $T^{2}$ is disjoint from $(\mathop\cup\limits_{i=1}^{n-1}F_{i})\cup(\mathop\cup\limits_{j=1}^{n}S_{j})$.
	
	

	 Then $T^{2}$ lies in some $V_{i}$ or $W_{i}$, for some $1\leq i\leq n$. Without  loss of generality,  we assume that $T^{2}$ lies in  $V_{i}$. Since $T^{2}$ is essential in $E(K^{\star})$, $i\geq 2$. It is known that there is no essential closed surface in a compression body or handlebody. So $T^{2}$ is isotopic to  $F_{i-1}$. Therefore $V^{\star}\cup_{S^{\star}}W^{\star}$ is an amalgamation of a Heegaard splitting of $E(K)$ and a Heegaard splitting of $C$ along $T^{2}$.
	 \begin{fact}
	 	$g(C)=2$.
	 \end{fact}
	 \begin{proof}
	  Since $p\geq 2$, $K^{\star}$ runs around the longitude of $\eta(K)$ at least twice. So $C$ is not a torus $I$-bundle.  On one hand, $\partial C$ consists of two tori. Then $g(C)\geq 2$. On the other hand, there is a genus two Heegaard splitting of $C=(T^{2}\times I)\cup_{A}\eta(K)$. Let $J=\overline{\eta(K)-(\partial \eta(K)\times I)}$ and $b$ be a fiber arc in $\overline{C-(T^{2}\times [0,\frac{1}{2}])-J}$ with one endpoint in $T^{2}\times\{\frac{1}{2}\}$ and the other in $\partial J$ and $b\cap int(A)\neq \emptyset$. Then both $V_{1}=(T^{2}\times [0,\frac{1}{2}])\cup \eta(b)\cup J$ and $W_{1}=\overline{C-V_{1}}$ are genus two compression bodies with $\partial_{+}V_{1}=\partial_{+}W_{1}$, see Figure 8. So $V_{1}\cup W_{1}$ is a genus two Heegaard splitting of $C$ and $g(C)\leq 2$. 
	  Hence $g(C)=2$.
	  \end{proof}
	  \begin{figure}[htbp]
	  	\centering
	  	\includegraphics[width=0.80\textwidth]{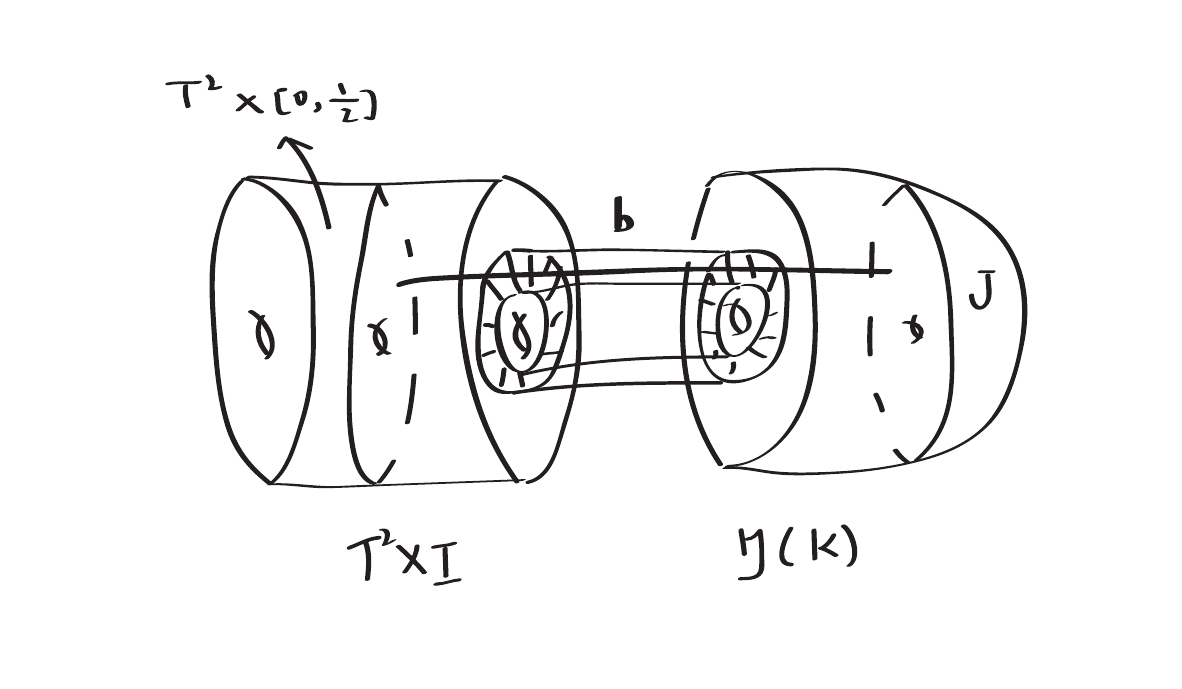}
	  	\caption{A genus two Heegaard splitting of $C$}
	  \end{figure}
	
	  Therefore, $g(K^{\star})=g(S^{\star})\geq g(K)+g(C)-g(T^{2})=g(K)+1$. And $t(K^{\star})\geq t(K)+1$. However,  $t(K^{\star})\leq t(K)+1$. Hence $t(K^{\star})=t(K)+1$.
	
  \subsection{Proof of Proposition 4.1 (2)}

     Recall that $InS$ is the collection of isotopy classes of  boundary slopes of properly embedded essential surfaces in $E(K)$. Then by Lemma \ref{lem3}, $InS$  contains finitely many vertices in $\mathcal {C}(T^{2})$ depending on $K$.  So  there is  a constant  $\mathcal{N}_{1}$ depending on $K$  so that $diam_{\mathcal{C}(T^{2})}(InS)\leq \mathcal{N}_{1}$.
    \begin{claim}\label{claim2}
    	$E(K)$ is not a twisted $I$-bundle over a compact non-orientable surface.
    \end{claim}

    \begin{proof}
    Suppose the conclusion is false. Then  $E(K)$ is a twisted $I$-bundle over a compact non-orientable surface $F$. Then $\chi(\partial E(K))=2\chi(F)$. Since $\partial E(K)$ is a torus, $\chi(\partial E(K))=2\chi(F)=0$. So $F$ is either a Mobius band or Klein bottle.  We say that $F$ is not a Mobius band. For if not,  then $E(K)$ is the twisted I-bundle of a Mobius band, i.e., a solid torus. So $K$ is a trivial knot in $S^{3}$.  A contradiction.  So $F$ is a Klein bottle. Let $\widetilde{F}$ be a double cover of $F$. Then  $\widetilde{F}$ is a torus and $\widetilde{F}\times I$, i.e.,  a torus I-bundle,   is a double covering of $E(K)$.    For any slope  $r \subset \partial E(K)$,  $\widetilde{F}\times I(r,r)$ is a double covering of $E(K)(r)$.
 However,   when $r$ is the meridian of $K$,  $\widetilde{F}\times I(m,m)$ is  a Lens space or $S^{3}$ while $E(K)(m)$ is $S^{3}$.  It means that $S^{3}$ is not simple connected.  A contradiction.  \end{proof}

    Let $\alpha\in InS$ be a boundary slope of an essential surface $Q$ properly embedded in $E(K)$ such that the genus of $Q$ is minimal. Let $g=max~\{g(Q),t(K)+2\}$. By Claim \ref{claim2}, $E(K)$ is not a twisted $I$-bundle over a non-orientable surface.  Since $K$ is nontrivial, $\partial E(K)$ is incompressible.  So $E(K)$ is not a product $I$-bundle over an orientable compact surface.  By Lemma \ref{lem4}, for any properly embedded, strongly irreducible and $\partial$-strongly irreducible,  genus at most $g$ surface $P$  in $E(K)$,  there is a number $K^{'}$ depending only on $g$ so that $d_{\mathcal{C}(T^{2})}(\partial P,\alpha)\leq K^{'}$. Let  $\mathcal{N}=K^{'}+\mathcal{N}_{1}$.

    Let  $V^{\star}\cup_{S^{\star}}W^{\star}$ be a minimal genus Heegaard splitting of $E(K^{\star})$ with $\partial_{-}V^{\star}=\partial E(K^{\star})$.  Since $E(K^{\star})$ is irreducible, $V^{\star}\cup_{S^{\star}}W^{\star}$ is either strongly irreducible or weakly reducible and irreducible.

    \begin{lemma}\label{lemma2}
    $V^{\star}\cup_{S^{\star}}W^{\star}$ is weakly reducible and irreducible.
    \end{lemma}

    \begin{proof}
    Suppose that the conclusion is false. Then $V^{\star}\cup_{S^{\star}}W^{\star}$ is strongly irreducible. Since $T^{2}$ is essential in $E(K^{\star})$,  $T^{2}$ intersects $S^{\star}$ nontrivially up to isotopy. By Claim \ref{claim1}, there is at most one strongly irreducible component in $S^{\star}\cap E(K)$ while others are essential in $E( K)$. Moreover, $\partial S^{\star}\cap E(K)$ are isotopic to the slope $p/q$ in $T^{2}$.
    
      If $S^{\star}\cap E(K)$ contains an essential surface in $E(K)$, then  $p/q\in InS$ and $diam_{\mathcal{C}(T^{2})}(p/q,InS)\leq \mathcal{N}_{1}\leq \mathcal{N}$. A contradiction. Otherwise, $S^{\star}\cap E(K)$ contains only one strongly irreducible component $S_{1}$.  By Lemma \ref{lem5},  either $S_{1}$ is strongly irreducible, $\partial$-irreducible or there is an 
 incompressible and boundary incompressible surface $F$ so that $d_{\mathcal{C}(T^{2})}(\partial S_{1},\partial F)\leq 1.$   For the first case,  $g(S_{1})\leq g(S^{\star}\cap E(K))\leq g(S^{\star})=g(K^{\star})=t(K^{\star})+1\leq t(K)+2\leq g$. Then $d_{\mathcal{C}(T^{2})}(\partial S_{1},\alpha)\leq K^{'}$, i.e., $d_{\mathcal{C}(T^{2})}(p/q,\alpha)\leq K^{'}$. Hence $diam_{\mathcal{C}(T^{2})}(p/q,InS)\leq K^{'}+\mathcal{N}_{1}=\mathcal{N}$. A contradiction.  For the second case, $diam_{\mathcal{C}(T^{2})}(p/q,InS)\leq \mathcal{N}_{1}+1\leq \mathcal {N}$. A contradiction.
    \end{proof}

    So $V^{\star}\cup_{S^{\star}}W^{\star}$ is weakly reducible and irreducible.  By the main result in \cite{STh}, $V^{\star}\cup_{S^{\star}}W^{\star}$ has an untelescoping $(V_{1}\cup_{S_{1}}W_{1})\cup_{F_{1}}...\cup_{F_{n-1}}(V_{n}\cup_{S_{n}}W_{n})$, $n\geq 2$ so that (1) for any $1\leq i\leq n-1$, $F_{i}$ is incompressible in $E(K^{\star})$ and not parallel to $\partial E(K^{\star})$; (2) for any $1\leq i\leq n$, $V_{i}\cup_{S_{i}}W_{i}$ is a strongly irreducible Heegaard splitting. We say that $T^{2}$ is disjoint from $(\mathop\cup\limits_{i=1}^{n-1}F_{i})\cup(\mathop\cup\limits_{j=1}^{n}S_{j})$. For if not, (1)  either $F_{i}\cap T^{2}\neq \emptyset$  up to isotopy for some $1\leq i\leq n-1$.  Then each component of $F_{i}\cap E(K)$ is essential in $E(K)$. Note that  $\partial F_{i}\cap E(K)$ are isotopic to  $p/q$ in  $T^{2}$. Then $diam_{\mathcal{C}(T^{2})}(p/q,InS)\leq \mathcal{N}_{1}\leq \mathcal{N}$. A contradiction. Or, (2)  $S_{i}\cap T^{2}\neq \emptyset$. By the similar argument in the proof of Lemma \ref{lemma2},   $diam_{\mathcal{C}(T^{2})}(p/q,InS)\leq \mathcal{N}_{1}\leq \mathcal{N}$.  A contradiction.
        
     Therefore $V^{\star}\cup_{S^{\star}}W^{\star}$ is an amalgamation of a Heegaard splitting of $E(K)$ and a Heegaard splitting of $C$ along $T^{2}$. So  $g(K^{\star})=g(S^{\star})\geq g(K)+g(C)-g(T^{2})=g(K)+1$ and $t(K^{\star})\geq t(K)+1$. However,  it is known that $t(K^{\star})\leq t(K)+1$. So $t(K^{\star})=t(K)+1$.


\end{proof}

\section{The proof of Corollary \ref{cor1}}
Let $K$ and $K^{\star}$ be the same as in Theorem \ref{thm1} (2).  For any  unstabilized Heegaard splitting $V^{\circ}\cup_{S^{\circ}}W^{\circ}$ of $E(K^{\star})$, 
since $E(K^{\star})$ is irreducible, it is either strongly irreducible or
weakly reducible and irreducible. So we divide it into two cases: (1) $V^{\circ}\cup_{S^{\circ}}W^{\circ}$ is strongly irreducible; (2)  $V^{\circ}\cup_{S^{\circ}}W^{\circ}$ is weakly reducible and irreducible.  We firstly prove Corollary \ref{cor1} for the strongly irreducible case.

Since $A$ is an essential annulus in $E(K^{\star})$,  by Schultens lemma \cite{Sch},  each curve of  $S^{\circ}\cap A$ is essential in both
$S^{\circ}$ and $A$.
Then by the same argument in the proof of Claim \ref{claim5.1}, one subsurface of  $\overline{S^{\circ}-A}$ is strongly irreducible  while the others are incompressible in their corresponding components  of $E(K^{\star})-A$.
Moreover, each subsurface of $\overline{S^{\circ}-A}$ has $p/q$ slopes as its boundary curves.    By the condition that
\begin{equation*}
	diam_{\mathcal{C}(T^{2})}(p/q,InS)>\mathcal{N},
	\end{equation*}
$S^{\circ}\cap E(K)$ is connected and strongly irreducible  while each component of $S^{\circ}\cap J$ is essential and  an annulus. 
So $S^{\circ}\cap J$ is a collection of nested annuli in $J$. By the similar argument of Case 1 of  Theorem \ref{theorem0}  in Section 3, we  do a Dehn surgery on $V^{\circ}\cup_{S^{\circ}}W^{\circ}$ and obtain a Heegaard splitting $V\cup_{S} W$ of $E(K)$.  In reverse,  $V^{\circ}\cup_{S^{\circ}}W^{\circ}$ is also a Dehn surgery of  $V\cup_{S} W$.

Otherwise, $V^{\circ}\cup_{S^{\circ}}W^{\circ}$ is weakly reducible and irreducible. Then it has an untelescoping, says $V^{\circ}\cup_{S^{\circ}}W^{\circ}=(V_{1}\cup_{S_{1}}W_{1})\cup_{F_{1}}...\cup_{F_{n-1}}(V_{n}\cup_{S_{n}}W_{n})$, for $n\geq 2$,  where  $F_{i}~(1\leq i\leq n-1)$ is incompressible in $E(K^{\star})$. By the same argument as above,  $A$ is disjoint from  $\mathop\cup\limits_{i=1}^{n-1}F_{i}$.  So $A$ is contained in $V_{i}\cup_{S_{i}}W_{i}$, for some $1\leq i\leq n$.  Then by the same argument as the strongly irreducible case, we do a Dehn surgery on $V_{i}\cup_{S_{i}}W_{i}$ so that it is changed into $V_{1,i}\cup_{S_{1,i}}W_{1,i}$.  Moreover, $(V_{1}\cup_{S_{1}}W_{1})\cup_{F_{1}}...\cup_{F_{i-1}} (V_{1,i}\cup_{S_{1,i}} W_{1,i})\cup_{F_{i}}...\cup_{F_{n-1}}(V_{n}\cup_{S_{n}}W_{n})$  is a Heegaard splitting of $E(K)$.   In reverse, $V^{\circ}\cup_{S^{\circ}}W^{\circ}=(V_{1}\cup_{S_{1}}W_{1})\cup_{F_{1}}...\cup_{F_{n-1}}(V_{n}\cup_{S_{n}}W_{n})$ is also a Dehn surgery of $(V_{1}\cup_{S_{1}}W_{1})\cup_{F_{1}}...\cup_{F_{i-1}} (V_{1,i}\cup_{S_{1,i}} W_{1,i})\cup_{F_{i}}...\cup_{F_{n-1}}(V_{n}\cup_{S_{n}}W_{n})$.

\bibliographystyle{amsplain}

\end{document}